\documentclass{article}

\usepackage{amsmath}
\usepackage{amsfonts}
\usepackage{amssymb}
\usepackage{amsthm}
\usepackage{amscd}
\usepackage{tikz-cd}
\usepackage[english]{babel}
\usepackage[square,sort,comma,numbers]{natbib}

\usepackage[shortlabels]{enumitem}
\usepackage{mathtools}

\theoremstyle{plain}
\newtheorem{theorem}{Theorem}
\newtheorem{lemma}[theorem]{Lemma}

\newtheorem{proposition}[theorem]{Proposition}
\newtheorem{definition}[theorem]{Definition}
\newtheorem{corollary}[theorem]{Corollary}

\theoremstyle{remark}
\newtheorem*{remark}{Remark}

\title{\textsc{Ring epimorphisms from path algebras to matrix algebras}}

\author{Jakub Kopřiva}

\begin{document}

	\maketitle
	
	\begin{abstract}
	    \noindent In this text, we are concerned with ring epimorphisms, and more specifically universal localisations, from path algebras to matrix algebras. We are mainly focused on constructing ring epimorphisms and universal localisations by extending them from from smaller path algebras to larger path algebras. At first, we discuss some simple results, and then we present generalizations thereof in several directions.
	\end{abstract}
	
\noindent \textbf{Key words:} Ring epimorphism, universal localisation, path algebra, matrix algebra, free associative algebra\\
\noindent \textbf{MSC2010:} 16S10, 16S50, 16S85 (Primary) 16G20 (Secondary)
	
	\tableofcontents
	
	\section*{Introduction}
	This document is structured into two sections: in section 1, we review, and, in section 2, we set out to construct ring epimorphisms from path algebras to matrix algebras. Our starting point is formed by Proposition \ref{Proposition19}, which yields a ring epimorphism to matrix algebra $M_n(k)$ arising a from a brick, a module with trivial endomorphisms, and Theorem \ref{Theorem20}, which shows that a ring epimorphism to $M_n(k\langle x_1, \dots, x_n \rangle)$ can be obtained, informally speaking, by adding an additional edge to a brick. If, moreover, the brick has no self-extensions, we obtain universal localisations.
	
	Our goal in this text is to apply the ideas underpinning the basic results in broader contexts. Among other results, we generalize Proposition \ref{Proposition19} and Theorem \ref{Theorem20} in Proposition \ref{PIdealRelations} and Theorem \ref{TEpirmorphismInvariant}, respectively, for ring epimorphisms that are not necessarily universal localisations. Moreover, in Proposition \ref{PExtendingExceptional}, we present a generalization of Theorem \ref{Theorem20} in terms of universal localisations using discussion on extending universal localisations in general.
	
	\section{Ring epimorphisms and universal localisations}
	
	In this section, we give an overview of ring epimorphisms and related results; we also derive some useful corollaries of these results.
	
	We begin by recalling the functors between categories of modules arising from a homomorphism between rings (all rings in this text are assumed to be unital); then, we define ring epimorphisms. The concept of ring epimorphism is examined from both ring-theoretic and category-theoretic point of view. Finally, we focus on universal localisations, specific type of ring epimorphisms that can be viewed as a vast, but natural, generalization of the concept of localisation in a multiplicative set in commutative algebra.
	
	Among other sources, this section draws on nice reviews of the subject in subsection 2.1 \textit{Ring epimorphisms} in \cite{marks2015ring} and subsection 2.1 \textit{Localisations of rings} in \cite{hugel2018flat}.
	
	\subsection{Change of rings}
	In this subsection, we briefly recapitulate some notions regarding change of rings induced by a homomorphism between them.
	
	Given a ring homomorphism $f: A \to B$, there are several functors between $\mathsf{Mod-}A$ and $\mathsf{Mod-}B$:
	\begin{enumerate}[(i)]
		\item \textit{restriction of scalars} functor $\varphi^*: \mathsf{Mod-}B \to \mathsf{Mod-}A$ such that: $$f^*(N_B) = N_A$$
		
		\item \textit{induced module} functor $f_!: \mathsf{Mod-}A \to \mathsf{Mod-}B$ such that: $$f_!(M_A) = M_A \otimes_A {}_A B_B$$

		\item \textit{coinduced module} functor $f_*: \mathsf{Mod-}A \to \mathsf{Mod-}B$ such that: $$f_*(M_A) = \mathrm{Hom}_A({}_B B_A, M_A)$$  
	\end{enumerate}
	These functors form two pairs of adjoint functors:
 	$$f_!: \mathsf{Mod-}A \dashv \mathsf{Mod-}B : f^* \mbox{  and  } f^*: \mathsf{Mod-}B \dashv \mathsf{Mod-}A : f_*$$
 	In other words, for any $M \in \mathsf{Mod-}A$ and $N \in \mathsf{Mod-}B$, we have the following functorial isomorphisms:
 	$$\mathrm{Hom}_B(M_A \otimes_A {}_A B_B, N_B) \cong \mathrm{Hom}_A(M_A, N_A),$$
 	$$\mathrm{Hom}_A(N_A, M_A) \cong \mathrm{Hom}_B(N_B, \mathrm{Hom}_A({}_B B_A,M_A)).$$
 	For details, we refer the reader to section 10.4 \textit{Tensor Product of Modules} in \cite{dummit2004abstract} for instance.
 	
	\subsection{Ring epimorphisms}
	In this section, we try to give a rather representative overview of results on ring epimorphisms and to illustrate the variety approaches with which this concept have been treated in the literature thus far.
	
	\begin{definition}[Ring epimorphism]\label{Definition1}
	We say that a ring homomorphism $\varphi: R \to S$ is a ring epimorphism if for any two ring homomorphisms $\varrho_1, \varrho_2: S \to T$ with $\varrho_1 \varphi = \varrho_2 \varphi,$ also $\varrho_1 = \varrho_2.$   	
	\end{definition}
	
	\begin{definition}[Dominion]\label{Definition2}
	Let $A \subseteq B$ be rings. The dominion of $A$ in $B$ is the maximal subset $D$ of $B$ such that homomorphisms from $B$ that agree on $A$ must agree on $D$.
	\end{definition}
	
	\begin{remark}
	It is easy to see that dominion $D$ needs to be a subring of $B$. For more details, see section 1 \textit{Preliminaries and corrections} in \cite{isbell1969epimorphisms} or section \textit{Maximal epic subrings and dominions in simple artinian rings} in chapter 7 in \cite{schofield1985representations}.
	\end{remark}

	\begin{theorem}[Ring-theoretic characterization of ring epimorphisms, Theorem 1.1 in \cite{isbell1969epimorphisms} attributed to Silver and Mazet and \cite{stenstroem2012rings}]\label{Theorem3}
	For a ring homomorphism $f: A \to B$, the following statements are equivalent:
	\begin{enumerate}[(i)]
	\item $f$ is a ring epimorphism;
	\item The dominion of $\varphi(A)$ in $B$ is equal to $B$.
	\item $f \otimes_A B = B \otimes_A f: B \to B \otimes_A B$ is an isomorphism of $B-$$B-$bimodules;
	\item $\mathrm{Coker}(f) \otimes_A B = 0$.
	\end{enumerate}
	Moreover, the dominion of $\varphi(A)$ is a set of all $XPY \in B$ where $X$ is a $1 \times n$ matrix over $B$, $Y$ is $n \times 1$ matrix over $B$, and $P$ is $n \times n$ matrix over $\varphi(A)$ such that $XP$ is a $1 \times n$ matrix over $\varphi(A)$ and $PY$ a $n \times 1$ matrix over $\varphi(A)$.
	\end{theorem}

	\begin{theorem}[Categorial characterization of ring epimorphisms, \cite{stenstroem2012rings}]\label{Theorem4}
	For a ring homomorphism $f: A \to B$, the following statements are equivalent:
	\begin{enumerate}[(i)]
	\item $f$ is a ring epimorphism;
	\item The restriction of scalars functor $f^*: \mathsf{Mod-}B \to \mathsf{Mod-}A$ (respectively the restriction of scalars from $B\mathsf{-Mod}$ to $A\mathsf{-Mod}$) is fully faithful.
	\end{enumerate}
	\end{theorem}
	
	\begin{theorem}[Categorial view of ring epimorphisms, Theorem 1.2 in \cite{gabriel1987quotients}, \cite{geigle1991perpendicular}, \cite{iyama2003rejective}]\label{Theorem5}
	There is a bijection between:
	\begin{enumerate}[(i)]
	\item ring epimorphisms $A \to B$ up to equivalence (given by isomorphisms of targets thereof);
	\item bireflective subcategories $\mathcal{X}_B$ of $\mathsf{Mod-}A$ (respectively, $A\mathsf{-Mod}$), i.e., strict full subcategories of closed under products, coproducts, kernels and cokernels.
	\end{enumerate}
	If, furthermore, $A$ is a finitely generated algebra over $k$, this bijection can be restricted between:
	\begin{enumerate}[(i)]
	\item $k-$algebra epimorphisms $A \to B$ up to equivalence, where $B$ is a finitely generated algebra over $k$;
	\item bireflective subcategories $\mathcal{X}_B$ of $\mathsf{mod-}A$ (respectively, $A\mathsf{-mod}$), i.e., strict full functorially finite subcategories closed under kernels and cokernels.
	\end{enumerate}
	\end{theorem}

	Now, we establish some useful consequences of the results above.
	
	\begin{corollary}[Ring epimorphisms and Morita equivalence]\label{Corollary7}
	Provided that $f: A \to B$ is a ring epimorphism and $A'$ is Morita equivalent to $A$, then there is a ring epimorphism $f': A' \to B'$ with $B'$  Morita equivalent to $B$.
	\end{corollary}
	\begin{proof}
	Let $F: \mathsf{Mod-}A \to \mathsf{Mod-}A'$ be an equivalence of categories, and let $\mathcal{X}_B$ be the bireflective subcategory corresponding to the ring epimorphism $f: A \to B$ by Theorem \ref{Theorem5}.
	
	The functor $F$ restricts to an equivalence between $\mathcal{X}_B$ and $(\mathcal{X}_B)'$ its essential image under $F$. As $F$ is an equivalence of categories, $(\mathcal{X}_B)'$ is strictly full subcategory of $\mathsf{Mod}-A'$, closed under products, coproducts, kernels and cokernels, hence bireflective, since $\mathcal{X}_B$ is such.
	
	From Theorem \ref{Theorem5}, we obtain that there is a ring epimorphism $f': A' \to B'$ such that the essential image of the restriction functor $(f')^*$ is $(\mathcal{X}_B)'$. We have that $\mathsf{Mod-}B$ is equivalent to $\mathcal{X}_B$, $\mathcal{X}_B$ to $(\mathcal{X}_B)'$, and $(\mathcal{X}_B)'$ to $\mathsf{Mod-}B'$. This yields that $\mathsf{Mod-}B$ and $\mathsf{Mod-}B'$ are equivalent.
	\end{proof}
	
	\begin{lemma}\label{Lemma7}
	Let $A$ be a ring, $M, N, L \in \mathsf{Mod-}A$, and 
	$$
	\begin{tikzcd}
	M \arrow[two heads]{d}{\varphi} \arrow{dr}{\psi} &\\
	L \arrow{r}{\varrho} & N
	\end{tikzcd}
	$$
	be a commutative square in $\mathsf{Set}$ with $\varphi, \psi$ $A-$homomorphisms, $\varphi$ surjective, then $\varrho$ is an $A-$homomorphism.
	\end{lemma}
	\begin{proof}
	Choose $\ell_1, \ell_2 \in L$ and $a \in A$. Since $\varphi$ is surjective, there are $m_1, m_2$ such that $\varphi(m_i) = \ell_i, i = 1,2$. Now, using commutativity of the diagram, we compute $\varrho(\ell_1 + \ell_2) = \varrho(\varphi(m_1 + m_2)) = \psi(m_1 + m_2) = \psi(m_1) + \psi(m_2) = \varrho(\ell_1) + \varrho(\ell_2)$. For $\varrho(\ell_1\cdot a) = \varrho(\ell_1) \cdot a$, we proceed similarly. 
	\end{proof}

	\begin{theorem}\label{Theorem8}
	A ring homomorphism $f: A \to B$ is a ring epimorphism if and only if, for all finitely generated $M, N \in \mathsf{Mod-}B$, $\mathrm{Hom}_A(M,N) = \mathrm{Hom}_B(M,N)$ and, for any $L \in \mathsf{Mod-}B$ and an $A-$homomorphism $\psi: B \to L$, the $\mathrm{Im}\,\psi$ is contained in a finitely generated $B-$submodule of $L$.
	\end{theorem}
	\begin{proof}
	$(\Rightarrow)$ The first part follows from Theorem \ref{Theorem3}. The second part from the fact that homomorphisms in $\mathrm{Hom}_A(B,L) = \mathrm{Hom}_S(B,L)$ correspond to cyclic $B-$submodules of $L$ via their image for any $L \in \mathsf{Mod-}B$.\\\\
	$(\Leftarrow)$ Let $M, N \in \mathsf{Mod-}B$ and $\tau: M \to N$ be an $A-$homomorphism between them. Choose $\varrho: B^{(I)} \twoheadrightarrow M$ surjective $B-$homomorphism. For each $i \in I$, there is a $N_i$ a finitely generated $B-$submodule such that $\mathrm{Im}\, \tau \varrho \mu_i$ lies inside it. Since, by our assumption, $\mathrm{Hom}_A(B,N_i) = \mathrm{Hom}_B(B,N_i)$, we have that $\tau \varrho \mu_i$ is an $B-$homomorphism. By universal property of $B^{(I)}$, there is a unique $B-$homomorphism $\psi: B^{(I)} \to N$ such that, for all $i \in I$, $\tau \varrho \mu_i = \psi \mu_i$. Due to its uniqueness, we get that $\tau \varrho = \psi$, we and apply Lemma \ref{Lemma7} to obtain that $\tau$ is an $B-$homomorphism. 
	\end{proof}

	\begin{corollary}\label{Corollary9}
	A ring homomorphism $f: A \to B$ that turns $B$ into a finitely generated $A-$module is a ring epimorphism if and only if, for all finitely generated $M, N \in \mathsf{Mod-}B$, $\mathrm{Hom}_A(M,N) = \mathrm{Hom}_B(M,N)$.
	\end{corollary}
	\begin{proof}
	For any $L \in \mathsf{Mod-}B$ and an $A-$homomorphism $\psi: B \to L$, the $\mathrm{Im}\,\psi$ is a finitely generated $A-$module since $B$ is such. A $B-$submodule of $L$ generated by these finitely many generators clearly contains $\mathrm{Im}\,\psi$.
	\end{proof}
	
	\subsection{Universal localisations}
	In this subsection, we gather some notions and results regarding universal localisations that will be of use later in this text.
	
	\begin{theorem}[Universal localisation, \cite{schofield1985representations}]\label{Theorem10}
	Let $\Sigma$ be a set of morphisms between finitely generated projective right $A-$modules. Then there are: a ring $A_\Sigma$ and a homomorphism of rings $f_\Sigma: A \to A_\Sigma$, called the universal localisation of $A$ at $\Sigma$, such that
	\begin{enumerate}[(i)]
	\item $f_\Sigma$ is $\Sigma-$inverting, i.e. if $\alpha: P \to Q$ belongs to $\Sigma$, then $\alpha \otimes_A A_\Sigma: P \otimes_A A_\Sigma \to Q \otimes_A A_\Sigma$ an isomorphism of right $A_\Sigma-$modules, and
	\item $f_\Sigma$ is universal $\Sigma-$inverting, i.e. for any $\Sigma-$inverting ring homomorphism $\psi: A \to B$, there is a unique ring homomorphism $\bar{\psi}: A_\Sigma \to B$ such that $\psi = \bar{\psi} f_\Sigma$.
	\end{enumerate}
	Moreover, the homomorphism $f_\Sigma$ is a ring epimorphism and $\mathrm{Tor}_1^A(A_\Sigma, A_\Sigma) = 0$. The restrictions of $A_\Sigma-$modules are modules $M \in \mathsf{Mod-}A$ such that, for each $\alpha \in \Sigma$, the map $\mathrm{Hom}_A(\alpha, M)$ is invertible. 
	\end{theorem}
	
	\begin{theorem}[Theorems 4.7 and 4.8 in \cite{schofield1985representations}]\label{Theorem11}
	Let $f: A \to B$ be a ring epimorphism; then the following conditions are equivalent:
	\begin{enumerate}[(i)]
		\item $\mathrm{Ext}^1_A = \mathrm{Ext}^1_B$ on $\mathsf{Mod-}B$;
		\item $\mathrm{Tor}_1^A(B,B)=0$;
		\item $\mathrm{Tor}_1^A = \mathrm{Tor}_1^B$ on $\mathsf{Mod-}B$;
		\item $\mathrm{Ext}^1_R = \mathrm{Ext}^1_S$ on $B\mathsf{-Mod}$
	\end{enumerate}
	Moreover, a universal localisation satisfies all these conditions.
	\end{theorem}

	\begin{remark}
	Provided that $f: A \to B$ a ring epimorphism turns $B$ into a finitely generated $A-$module and $A$ is Noetherian, using a finite free presentation of $B$ in the proof in \cite{schofield1985representations} (Theorem 4.8 therein) yields that $\mathrm{Ext}^1_A = \mathrm{Ext}^1_B$ on finitely generated $B-$modules implies $\mathrm{Tor}_1^A(B,B)=0$.
	\end{remark}

	\begin{definition}[\cite{geigle1991perpendicular}]\label{Definition12}
	A ring epimorphism $f: A \to B$ is called homological ring epimorphism if $\mathrm{Tor}^A_n(B,B) = 0$ for all $n \in \mathbb{N}$.
	\end{definition}

	\begin{theorem}[\cite{krause2010telescope}]\label{Theorem13}
	For hereditary rings, universal localisations and homological epimorphisms coincide.
	\end{theorem}

	Similarly as in the previous section, we now derive some useful conclusions from the results above:
	
	\begin{corollary}[Universal localisations and Morita equivalence]\label{Corollary14}
	Provided that $f: A \to B$ is a universal localisation and $A'$ is Morita equivalent to $A$; then there is a universal localisation $f': A' \to B'$ with $B'$ Morita equivalent to $B$.
	\end{corollary}
	\begin{proof}
	Let $B = A_\Sigma$ for some $\Sigma$, and let $F: \mathsf{Mod-}A \to \mathsf{Mod-}A'$ be an equivalence of categories. As $F$ sends finitely generated projectives in $\mathsf{Mod-}A$ to finitely generated projectives in $\mathsf{Mod-}A$, $F(\Sigma)$ is also a set of maps between finitely generated projectives. Let $f': A' \to B'$ be the ring epimorphism by Corollary \ref{Corollary9} such that $B'$ is Morita equivalent to $A_\Sigma$. Since $F$ restricted between $\mathcal{X}_{A_\Sigma}$ and $\mathcal{X}_{B'}$ is an equivalence, then $f'$ is $F(\Sigma)-$inverting. The fact that $f'$ is universal $F(\Sigma)-$inverting is then proved by going back to $\mathsf{Mod-}A$ by equivalence $F$ and using that $f$ is universal $\Sigma-$inverting. Consequently, $B' = A'_{F(\Sigma)}$.
	\end{proof}
	
	\begin{lemma}\label{Lemma15}
		Let $A \subseteq B$ be rings, and $\Sigma$ be a set of maps between finitely generated projectives over $A$. Denote $\Sigma' = \{\alpha \otimes_A B \,|\, \alpha \in \Sigma\}$; this is a set of maps between finitely generated projectives over $B$. If a $B-$module $N$ can be viewed as a restriction of a $B_{\Sigma'}-$module, then $N_A$ is a restriction of a $A_{\Sigma}-$module.
	\end{lemma}
	\begin{proof}
	By Theorem \ref{Theorem10}, $N$ can be viewed as a restriction of a $B_{\Sigma'}-$module if $\mathrm{Hom}_B(\alpha', N)$ is an isomorphism for all $\alpha' \in \Sigma'$. Write $\alpha'$ as $\alpha \otimes_A B: P_1 \otimes_A B \to P_2 \otimes_A B$ for $P_i$ finitely generated projectives over $A$.
	
	The adjuction of induced module and restriction of scalars functors yields functorial isomorphisms:
	$$\mathrm{Hom}_B(P_i \otimes_A B, N) \cong \mathrm{Hom}_A(P_i, N_A)$$
	Provided that $\mathrm{Hom}_B(\alpha', N)$ is an isomorphism, the map $\mathrm{Hom}_A(\alpha, N_A)$ needs to be an isomorphism as well.
	
	Thus, if $\mathrm{Hom}_B(\alpha', N)$ is an isomorphism for all $\alpha' \in \Sigma'$ (and $N$ is a restriction of a $B_{\Sigma'}-$module), then $\mathrm{Hom}_A(\alpha, N_A)$ is an isomorphism for all $\alpha \in \Sigma$ (and $N_A$ is a restriction of a $A_{\Sigma}-$module).
	\end{proof}
	
	\begin{lemma}\label{Lemma16}
	Let $\Sigma$ be a set of morphisms between finitely generated projective right $A-$modules and $I \subseteq A$ be an ideal of $A$. Then $A_\Sigma/\langle I \rangle_{A_\Sigma}$ and $(A/I)_{\Sigma \otimes A/I}$ are isomorphic as rings where $\Sigma \otimes A/I = \{\alpha \otimes_A A/I \,|\, \alpha \in \Sigma\}$
	\end{lemma}
	\begin{proof}
	We prove this by showing that both $A'_1 = A_\Sigma/\langle I \rangle_{A_\Sigma}$ and $A'_2=(A/I)_{\Sigma \otimes A/I}$ are pushouts of the following diagram: $A/I \leftarrow A \rightarrow A_\Sigma$; hence they are isomorphic as rings.
	
	We begin our discussion with $A'_1$. Naturally, we have a projection $\pi_{\langle I \rangle_{A_\Sigma}}: A_\Sigma \to A'_1$. If we denote $f_\Sigma: A \to A_\Sigma$ the universal localisation, we have a map $\pi_{\langle I \rangle_{A_\Sigma}} \circ f_\Sigma: A \to A'_1$. Clearly, this map maps elements of $I$ to zero. By the homomorphism theorem, there exists a unique map $g: A/I \to A'_1$ such that $g \circ \pi_I = \pi_{\langle I \rangle_{A_\Sigma}} \circ f_\Sigma$. Suppose we have two maps $h_1: A/I \to B$ and $h_2: A_\Sigma \to B$ such that $h_1 \circ \pi_I = h_2 \circ f_\Sigma$. Since $h_2 \circ f_\Sigma$ maps all elements of $I$ to zero, $h_2$ needs to map all elements of $\langle I \rangle_{A_\Sigma}$ to zero. This gives us, by the homomorphism theorem, that there is a unique homomorphism $h'$ such that $h_2 = h' \circ \pi_{\langle I \rangle_{A_\Sigma}}$. Now, we have two ring homomorphisms $h_1$ and $h' \circ g$ from $A/I$ to $B$; they agree on $A$, since $h_1 \circ \pi_I = h_2 \circ f_\Sigma = (h' \circ \pi_{\langle I \rangle_{A_\Sigma}}) \circ f_\Sigma = (h' \circ g) \circ \pi_I$. As $\pi_I$ is a ring epimorphism, $h_1$ and $h' \circ g$ need to be equal. The uniqueness of $h'$ is warranted by the universal property of the quotient $A_\Sigma/\langle I \rangle_{A_\Sigma}$.

	Second, let us have $A'_2$. We have a universal localisation $f_{\Sigma \otimes_A A/I}: A/I \to A'_2$; we proceed in an analogous manner. Given that $f_{\Sigma \otimes_A A/I}$ is $\Sigma-$inverting, by Theorem \ref{Theorem8} stating the universal property of universal localisation, there exists a unique ring homomorphism $g: A_\Sigma \to A'_2$ such that $g \circ f_\Sigma = f_{\Sigma \otimes_A A/I} \circ \pi_I$. Assume we have two maps $h_1: A/I \to B$ and $h_2: A_\Sigma \to B$ such that $h_1 \circ \pi_I = h_2 \circ f_\Sigma$. This gives us that $h_1$ is $\Sigma \otimes_A A/I-$inverting yielding a unique ring homomorphism $h': A'_2 \to B$ arising from its universal property as a universal localisation such that $h_1 = h' \circ f_{\Sigma \otimes_A A/I}$. Again, we have two maps $h_2$ and $h' \circ g$ from $A_\Sigma$ to $B$. However, as there two homomorphism agree when precomposed with $f_\Sigma$, the fact that $f_\Sigma$ is a ring epimorphisms gives us that, indeed, $h' \circ g = h_2$.
	\end{proof}

	\begin{remark}
	Provided that $A$ is an $R-$algebra, $A_\Sigma/\langle I \rangle_{A_\Sigma}$ and $(A/I)_{\Sigma \otimes_A A/I}$ are isomorphic as $R-$algebras.	
	\end{remark}
	
	\section{Ring epimorphisms from path algebras to matrix rings}
	
	In this section, we discuss ring epimorphisms, and more specifically universal localisations, from a finite-dimensional path algebra of quiver $Q$ over $k$, an algebraically closed field, to matrices over algebras over $k$, particularly of type $k\langle x_1, \dots, x_n \rangle$.
	
	This section is structured as follows. At first, we recall some preliminary results. Subsequently, we look at an example of a ring epimorphism (universal localisation) in question arising from bricks (exceptional modules) over path algebra $kQ$ and their simple extensions. Inspired by these examples of constructing ring epimorphisms and universal localisations from smaller to larger path algebras, we provide generalizations of them in several directions.
	
	The quiver $Q$ is a finite acyclic quiver $Q = (Q_0, Q_1, t, h)$. For brevity, we denote the corresponding path algebra $A = kQ$ in the first two subsections. We note that $A$ is a finite dimensional hereditary algebra and that $\mathsf{Rep}_k(Q)$ and $\mathsf{Mod-}A$ (or respectively, $\mathsf{rep}_k(Q)$ and $\mathsf{mod-}A$) are naturally equivalent. For foundations of the theory of representations of quivers and finite-dimensional algebras, we refer the reader to \cite{assem2006elements}.

	\subsection{Preliminaries}
	
	In this section, we frequently work with matrix rings $M_n(R)$ over a given ring $R$. For this end, we establish some notation. We set:
	$$e_{ij} = \left\{\begin{array}{cl} 1 & \mbox{at position } ij; \\ 0 & \mbox{elsewhere.} \end{array}\right.$$
	Recall that $e_{ij} e_{k \ell} = \delta_{jk} \cdot e_{i \ell}$ and that $e_{ii} A  e_{jj} = a_{ij} \cdot e_{ij}$ for any $A \in M_n(R)$.

	\begin{definition}[\cite{cohn2006free}]\label{Definition17}
		Given a ring $R$, we say that:
		\begin{enumerate}[(i)]
			\item $R$ is projective-free if every finitely generated projective module is free and of unique rank	;
			\item $R$ is a free ideal ring (fir) if all one-sided ideals of $R$ are free as $R-$modules;
			\item $R$ is a semifir if all finitely generated one-sided ideals of $R$ are free as $R-$modules.
		\end{enumerate}
	\end{definition}

	\begin{remark}
	 All semifirs are projective-free by Corollary 2.3.4 in \cite{cohn2006free}.	
	\end{remark}
	
	\begin{remark}
	For any $n$, the free associative algebra $k\langle x_1, \dots, x_n \rangle$ is a fir. See Corollary 2.5.2 in \cite{cohn2006free} and ensuing discussion therein.
	\end{remark}
	
	Using these notions and results, we formulate, and prove some auxiliary results for this section.
	
	\begin{lemma}\label{Lemma18}
		Let $S$ be a projective-free ring, and $\varphi: R \to M_n(S)$ be a ring homomorphism. Given a full set of orthogonal idempotents $e_1, \dots, e_m \in R$, there exists an automorphism $\psi$ of $S$ such that, for each $i = 1, \dots, m$, the image $\psi \varphi(e_i)$ is equal to:
		$$\psi \varphi(e_i) = \sum_{j} e_{jj}.$$
	\end{lemma}
	\begin{proof}
		If $e_1, \dots, e_m \in R$ form a full set of orthogonal idempotents, then their imagages, $\varphi(e_1), \dots, \varphi(e_m) \in M_n(S)$, also form such a set. Every $\varphi(e_i)$ yields a direct summand of $S^n$; since $S$ is projective-free, each such summand is a free $S-$module of a unique rank $r_i$. We also have $A_i$ an $n \times r_i$ matrix over $S$ and $B_i$ an $r_i \times n$ matrix over $S$ such that $A_iB_i = \varphi(e_i)$ and $B_iA_i = I_{r_i}$, the $r_i \times r_i$ identity matrix. The matrix $A_i$ gives the injection from $S^{r_i}$ to $S^n$; whereas the matrix $B_i$ represents the projection from $S^n$ to $S^{r_i}$.
		
		As the idempotents $\varphi(e_1), \dots, \varphi(e_m) \in M_n(S)$ are orthogonal, we obtain that $B_j A_i = 0$ for $i \neq j$. This yields that:
		$$\left(\!\! \begin{array}{c} B_1 \\ \dots \\ B_m \end{array} \!\!\right) \left(\!\! \begin{array}{lcr} A_1 & \!\! \dots \!\! & A_m \end{array} \!\!\right) = I_n.$$
		The identity matrix $I_n$ should be thought of a block diagonal matrix with identity matrices $I_{r_1}, \dots, I_{r_m}$ on the diagonal. As $S^n$ is projective, the image of $B^T = \left(\!\! \begin{array}{lcr} B_1 & \!\! \dots \!\! & B_m \end{array} \!\!\right)^T$ splits; we obtain that $S^n \cong P \oplus S^n$. Since $S$ is projective-free, the module $P$ needs to be free; however, due to unique rank of $S^n$ the module $S^n$ is zero. Thus, the matrices $B$ and $A = \left(\!\! \begin{array}{lcr} A_1 & \!\! \dots \!\! & A_m \end{array} \!\!\right)$ give mutually an $S-$automorphisms of $S^n$.
				
		Therefore, we may set $\psi$ to be the inner automorphism of $M_n(S)$ given by congujation by $A^{-1}$. We calculate $\psi \varphi(e_i)$ as follows:
		$$\psi \varphi(e_i) = A^{-1} \varphi(e_i) A = \left(\!\! \begin{array}{c} B_1 \\ \dots \\ B_m \end{array} \!\!\right) A_i B_i \left(\!\! \begin{array}{lcr} A_1 & \!\! \dots \!\! & A_m \end{array} \!\!\right) = \left(\!\! \begin{array}{lcr} 0 & 0 & 0 \\ 0 & I_{r_i} & 0 \\ 0 & 0 & 0 \end{array} \!\!\right)$$
		Where the matrix on the right is to be understood as a block matrix.
	\end{proof}
	
	\begin{remark}
	Suppose $f: A \to M_m(k\langle x_1, \dots, x_n \rangle)$ is a ring homomorphism. Then, by the lemma above, there exists a dimension vector $\alpha$ such that for any finite-dimensional module $N$ over $M_m(k\langle x_1, \dots, x_n \rangle)$, the dimension vector of $N_A$ is a multiple of $\alpha$. The values of $\alpha$ are given	by unique ranks of free $k\langle x_1, \dots, x_n \rangle-$modules given by images orthogonal idempotents $f(e_v)$ for $v \in Q_0$.
	
	Hence, we can write $f: A \to M_\alpha(k\langle x_1, \dots, x_n \rangle)$ for a ring homomorphism from $A$ to matrices over $k\langle x_1, \dots, x_n \rangle$ such that the $A-$modules that lie in the image of $f^*$ have a multiple of $\alpha$ as their dimension vector.
		
	One could also establish this fact by means of algebraic geometry. We note that, for any finite-dimensional module $N$ over the algebra $M_m(k\langle x_1, \dots, x_n \rangle)$, we have that $\mathrm{dim}_k N = \ell m$ for some $\ell \geq 0$. Also, up to isomorphism, any such $N$ is parameterized by a choice of $X_1, \dots, X_n \in M_\ell(k)$ that $g_{X_1, \dots, X_n}: M_m(k\langle x_1, \dots, x_n \rangle) \to M_{\ell m}(k)$ maps $g_N(U)_{ij} = U_{ij}(X_1, \dots, X_n)$ for all $U \in M_m(k\langle x_1, \dots, x_n \rangle)$. The restriction $f^*$ is represented by $g_{X_1, \dots, X_n} \circ f$ in this setting.

	This means that each representation of $M_m(k\langle x_1, \dots, x_n \rangle)$ can be thought of as a point in $\mathbb{A}_{k}^{n\cdot \ell^2}$ representing the choice of $n$ matrices $\ell \times \ell$ over $k$. Now, for each $v \in Q$, the set of points of $(Y_1, \dots, Y_n) \in \mathbb{A}_{k}^{n\cdot \ell^2}$ such that $g_{Y_1, \dots, Y_n}[f(e_v)]$ has maximal rank $\alpha_v$ is dense open. Naturally, the set of points $\mathbb{A}_{k}^{n\cdot \ell^2}$ at which each of $f(e_v)$ are simultaneously mapped to a matrix with maximum possible rank $\alpha_v$ is dense open. This means that $\sum_{v \in Q_0} \alpha_v = \ell m$, as $N = \bigoplus_{v \in Q_0} g_N f(e_v)N$ as a vector space and that there no points in $\mathbb{A}_{k}^{n\cdot \ell^2}$ with any $f(e_v)$ mapped to a matrix of rank smaller than $\alpha_v$ as they would need to be of lower dimension than $\ell m$.
	\end{remark}
	
	\subsection{Simple examples}
	In this subsection, we present some simple results on ring epimorphisms and universal localisations from path algebras to matrix algebras over $k\langle x_1, \dots, x_n \rangle$. Later in this text, we will strive to find generalizations of these results.
	\begin{proposition}\label{Proposition19}
	Let $B \in \mathsf{mod-}A$	be a module, then the following are equivalent:
	\begin{enumerate}[(i)]
		\item $B$ is a brick, i.e. $\mathrm{End}_{A}(B) \cong k$, $($\!and $\mathrm{Ext}_{A}^1 (B,B)$ is zero$)$,
		\item The ring homomorphism $f_B: A \to \mathrm{End}_k(B)$ is a ring epimorphism $($\!universal localisation$)$.
	\end{enumerate}
\end{proposition}
\begin{proof}
	$(i)$ $\Rightarrow$ $(ii)$: Given $B \in \mathsf{mod-}A$ a brick over $A$, we have a natural ring homomorphism $f_B: A \to \mathrm{End}_k(B)$ giving $B_k$ as a vector space the structure of an $A-$module. Denote $\Lambda = \mathrm{End}_k(B)$. In accordance with Theorem \ref{Theorem4}, we prove that $(f_B)^*$ is fully faithful. By Corollary \ref{Corollary9}, as $f_B$ makes $\Lambda$ a finitely generated module over $A$, it suffices to show that $(f_B)^*$ is fully faithful on finite-dimensional modules.

	Clearly, the essential image of $(f_B)^*$ on finitely generated $\Lambda-$modules are only direct sums of finitely many copies of $B$ up to isomorphism. Let $I, J$ be finite sets, then we have:
	$$\mathrm{Hom}_{A}(B^I, B^J) = \prod_{I, J} \mathrm{Hom}_{A}(B,B) = \prod_{I, J} \mathrm{Hom}_{\Lambda}(B,B) = \mathrm{Hom}_{\Lambda}(B^I, B^J).$$
	Therefore, $f_B: A \to \Lambda$ is a ring epimorphism. If it is also given that $\mathrm{Ext}_{A}^1 (B,B) = 0$, then by Remark after Theorem \ref{Theorem11} (as $A$ is right Noetherian and $\Lambda$ is a finitely generated module over $A$) we obtain that $\mathrm{Tor}^{A}_1(\Lambda, \Lambda) = 0$ (all higher $\mathrm{Tor}$ groups vanish since $A$ is hereditary). This makes $f_B$ a homological ring epimorphism, and, therefore, a universal localisation by Theorem \ref{Theorem13}.\\\\
	$(ii)$ $\Rightarrow$ $(i)$: Clearly, $B$ is a brick over $\Lambda$, and, thus, over $A$ since $f_B$ is a ring epimorphism. Provided that $f_B$ is, furthermore, a universal localisation, then $\mathrm{Ext}^1_{A}(B, B) = \mathrm{Ext}^1_\Lambda(B, B) = 0$ by Theorem \ref{Theorem11} and since all non-trivial $\mathrm{Ext}$ groups vanish for $\Lambda$. 
\end{proof}

\begin{theorem}\label{Theorem20}
	Let $B \in \mathsf{mod-}A$ be an exceptional module of dimension vector $\alpha$; $Q'$ be a quiver such that $Q_0 = Q'_0$ and $Q_1 \subseteq Q'_1$, and $A' = kQ'$. Set $n = \sum_{e \in Q'_1 - Q_1} \alpha_{s(e)}\alpha_{t(e)}.$ There exists a universal localisation $f'_B: A' \to M_{\alpha}(k\langle x_1, \dots, x_n \rangle)$ that restricts to the universal localisation $f_B: A \to \mathrm{End}_k(B)$.
\end{theorem}
\begin{proof}
	Let denote $f_B$ the universal localisation of $A$ in $\Sigma$, and set $\Sigma' = \{\sigma \otimes_A A': P \otimes_A A' \to Q \otimes_A A'\,|\,\sigma: P \to Q \in \Sigma\}$. Modules over the universal localisation $A'_{\Sigma'}$ among modules over $A$ are those for which $\mathrm{Hom}_{A'}(\sigma \otimes_A A', M)$ are invertible for all $\sigma \otimes_A A' \in \Sigma'$. 
	
	Using Lemma \ref{Lemma15}, we obtain that for $\sigma \in \Sigma$, $\mathrm{Hom}_{A'}(\sigma \otimes_A A', M)$ is invertible if and only if $\mathrm{Hom}_A(\sigma, M_A)$ is invertible. Naturally, $M_A$ is the image of $M$ under the restriction of scalars functor induced by the inclusion $A \subseteq A'$.
	
	From Proposition \ref{Proposition19}, we have a universal localisation $f_B: A \to M_{\alpha}(k)$. By change of basis, we may assume that $f_B(e_v)$ maps to a matrix with ones at positions $(n_v, n_v), (n_v + 1, n_v + 1), \dots, (n_v + \alpha_v, n_v + \alpha_v)$ and zeros elsewhere. We now define a $k-$algebra homomorphism from $f'_B: A' \to M_{\alpha}(k\langle x_1, \dots, x_n \rangle)$ as extension of $f_B$. It suffices to define the image of $e \in Q'_1 - Q_1$; it is defined as follows: $$[f'_B(e)]_{ij} = \left\{ \begin{array}{ll} x_{ij, e} & n_{s(e)} \leq i \leq n_{s(e)} + \alpha_{s(e)}, n_{s(e)} \leq j \leq n_{s(e)} + \alpha_{s(e)}\\ 0 & \mbox{otherwise} \end{array} \right.$$
	We observe that $f'_B(e_{s(e)}) \cdot f'_B(e) \cdot f'_B(e_{s(e)})$ for all $e \in Q'_1 - Q_1$. Due to the universal property of path algebras (Theorem 1.8 in chapter II in \cite{assem2006elements}), $f'$ is a $k-$algebra homomorphism.
	Since $f_B$ is a ring epimorphisms, the dominion $D$ of $f'_B(A')$ in $M_{\alpha}(k\langle x_1, \dots, x_n \rangle)$ contains $M_{\alpha}(k)$. To prove that $D$ equals $M_{\alpha}(k\langle x_1, \dots, x_n \rangle)$, it suffices to show that $x_\ell \cdot I_\alpha \in D$ for all $1 \leq \ell \leq n$. For any $\ell$, $x_\ell$ corresponds to $x_{i_\ell j_\ell, e_\ell}$ for some $1 \leq i_\ell, j_\ell \leq \alpha$ and $e_\ell \in Q'_1 - Q_1$. We obtain $$\left(\!\! \begin{array}{ccc} x_\ell & &\\ & \ddots & \\  & & x_\ell \end{array} \!\!\right) = \sum_{i = 1}^n e_{i, i_\ell} \cdot f'_B(e_\ell) \cdot e_{j_\ell, i},$$ which lies in $D$ as it is a subring. Matrices $e_{ij}$ belong to $D$ and so does $f'_B(e_\ell)$.
	
	Finally, we show that all $A'_{\Sigma'}-$ modules are actually $M_{\alpha}(k\langle x_1, \dots, x_n \rangle)$ modules via $f'_B$. Any module $M$ that is an $A'_{\Sigma'}-$module is $M_\alpha(k)-$module via the universal localisation $f_B$. Therefore as a representation, $M$ is isomorphic to $[(k^{(I)})^{\alpha_v}, \varphi_e]_{v \in Q_0, e \in Q'_1}$ where $\varphi_e \in \mathrm{Hom}((k^{(I)})^{\alpha_s(e)}, (k^{(I)})^{\alpha_s(e)})$. We can think of it as a $n_{s(e)} \times n_{s(e)}$ matrix with entries in $\mathrm{End}_k(k^{(I)})$. Naturally, $[(k^{(I)})^{\alpha_v}, \varphi_e]_{v \in Q_0, e \in Q'_1}$ is a $M_{\alpha}(k\langle x_1, \dots, x_n \rangle)-$module as $x_\ell$ acts as the $(i_\ell - n_{t(e_\ell)}, j_\ell - n_{s(e_\ell)})$ entry of $\varphi_{e_\ell}$.
\end{proof}

\begin{remark}
The construction of a universal localisation $A' \to M_{n}(k)$ based on a universal localisations $A \to M_n{k}$ arising from a brick $B$ with no self-extension performed in the theorem above can be carried out in a similar manner even if $B$ has self-extensions. However, the maps are only ring epimorphisms, not universal localisations in such a case.
\end{remark}

\subsection{Extending ring epimorphisms}
In this subsection, we aim to generalize Theorem \ref{Theorem20} by constructing ring epimorphisms, without requiring that they be universal localisations, from path algebras to matrices over $k\langle x_1, \dots, x_n\rangle$ using a well-suited representations thereof.

\begin{lemma}\label{LAdditiveFunctorFull}
Let $F: \mathsf{A} \to \mathsf{B}$ be an additive functor between additive categories, then $F$ is full if and only if for every $A \in \mathsf{A}$ the group homomorphism $F_{A,A}: \mathrm{Hom}_{\mathsf{A}}(A,A) \to \mathrm{Hom}_{\mathsf{B}}(F(A),F(A))$ is surjective.
\end{lemma}
\begin{proof} Suppose that $F_{A,A}$ is surjective for every $A \in \mathsf{A}$, and let us have $f' \in \mathrm{Hom}_{\mathsf{B}}(F(A'),F(B'))$ for some $A', B' \in \mathsf{A}$.

Consider the biproduct $A' \oplus B'$ in $\mathsf{A}$ with morphisms $\iota_{A'}: A' \to A' \oplus B', \pi_{A'}: A' \oplus B' \to A'$ and $\iota_{B'}, \pi_{B'}$ similarly.

From $f'$, we obtain $F(\iota_{B'}) \circ f' \circ F(\pi_{A'}): F(A') \oplus F(B') \to F(A') \oplus F(B')$ in $\mathsf{B}$. Since $F$ is additive, $F(A') \oplus F(B') \cong F(A' \oplus B')$. Our assumptions gives us that there is $f: A' \oplus B' \to A' \oplus B'$ such that $F(f) = F(\iota_{B'}) \circ f' \circ F(\pi_{A'})$. It is easy to see that $F(\pi_{B'}) \circ F(f) \circ F(\iota_{A'}) = f'$. Therefore, $F$ is full.
\end{proof}

\begin{theorem}\label{TEpimorphismIdeal}
Let $f: A \to M_n(B)$ be a homomorphism of $k-$algebras. Then $f$ is a ring epimorphism if and only if the ideal $I_f$ of $B \ast_k k\langle (v_{ij})_{i,j=1}^n \rangle$ generated by $(v_{ij})_{i,j=1}^n \cdot f(a) - f(a) \cdot (v_{ij})_{i,j=1}^n$ for $a \in A$ understood element-wise contains $v_{ii} - v_{jj}$, $v_{i'j'}$, and $b v_{ii} - v_{ii} b$ for all $i,j,i' \neq j' \in \{1, \dots, n\}$ and $b \in B$.
\end{theorem}
\begin{proof}
$(\Leftarrow)$ Suppose that $I_f$ of $B \ast_k k\langle (v_{ij})_{i,j=1}^n \rangle$ contains $v_{ii} - v_{jj}$, $v_{i'j'}$, and $b v_{ii} - v_{ii} b$ for all $i,j,i' \neq j' \in \{1, \dots, n\}$ and $b \in B$. We will show that every $A-$endomorphism of a right $M_n(B)-$module is already an $M_n(B)-$endomorphism. Combining this with Theorem \ref{Theorem4} and Lemma \ref{LAdditiveFunctorFull} above will give us that $f$ is a ring epimorphism.

Any $M_n(B)-$module can be obtained as follows: given a vector space $W$ over $k$ and a homomorphism $\mu: B \to \mathrm{End}_k (W)$, we define a right $M_n(B)-$module structure on $W^n$ such that $b \in B$ act diagonally and $e_{ij}$ via itself. The structure of an $A-$module on $W^n$ is induced by restriction of scalars per $f: A \to M_n(B)$.

Suppose $(w_{ij})_{i,j=1}^n$ is an $A-$endomorphism of $W^n$; this means that for all $a \in A$ we have $(w_{ij})_{i,j=1}^n \mu(f(a)) - \mu(f(a)) (w_{ij})_{i,j=1}^n = 0$. There is a unique homomorphism from $B \ast_k k\langle (v_{ij})_{i,j=1}^n \rangle$ to $\mathrm{End}_k (V)$ defined by $\mu$ on $B$ and by $v_{ij} \mapsto w_{ij}$ on $k\langle (v_{ij})_{i,j=1}^n \rangle$. Moreover, since $(w_{ij})_{i,j=1}^n$ is an $A-$endomorphism of $W^n$, this homomorphism factors through $B \ast_k k\langle (v_{ij})_{i,j=1}^n \rangle / I_f$. As $v_{ii} - v_{jj}, v_{i'j'} \in I_f$ and $b v_{ii} - v_{ii} b \in I_f$ for all $i,j,i' \neq j' \in \{1, \dots, n\}$ and $b \in B$, $(w_{ij})_{i,j=1}^n$ clearly is also an $M_n(B)-$endomorphism.\\\\
$(\Rightarrow)$ Assume now that $f: A \to M_n(B)$ is an epimorphism. Denote $V = B \ast_k k\langle (v_{ij})_{i,j=1}^n \rangle / I_f$ as vector space over $k$. Consider a homomorphism $B \to \mathrm{End}_k (V), b \mapsto \mu(b)$, where $\mu(b)$ is the linear map of on $V$ such that $v \mapsto v \cdot b$; this gives a right $M_n(B)-$module structure on $V^n$. By definition, $(\mu(v_{ij}))_{i,j=1}^n$ is a well-defined endomorphism of $V^n$ as a right $A-$module.

Because $f$ is a ring epimorphism, $(\mu(v_{ij}))_{i,j=1}^n$ is also an endomorphism of $V^n$ as a right $M_n(B)-$module. This means that $\mu(v_{ii}) - \mu(v_{jj}) = 0$, $\mu(v_{i'j'}) = 0$, and $\mu(v_{ii})\mu(b) - \mu(b) \mu(v_{ii}) = 0$ for all $i,j,i' \neq j' \in \{1, \dots, n\}$ and $b \in B$. Applying these maps to $1 \in V$ yields the claim.
\end{proof}
The following proposition may be seen as a generalization of Proposition \ref{Proposition19} above, and it will be used in a subsequent proof of a generalization of Theorem \ref{Theorem20} is a similar way too.
\begin{proposition}\label{PIdealRelations}
Suppose that $A$ is finite-dimensional algebra over $k$ and $M \in \mathsf{mod}-A$ of dimension $n$. The $A-$module structure on $M$ is given by a homomorphism $f: A \to M_n(k)$; this homomorphisms gives us $I_f$ an ideal of $k \ast_k k\langle (v_{ij})_{i,j = 1}^n \rangle \cong k\langle (v_{ij})_{i,j = 1}^n \rangle$. If $\sum_{i,j = 1}^n \lambda_{ij} m_{ij} = 0$ for all $(m_{ij})_{i,j = 1}^n$ $A-$endomor- phisms of $M_A \cong k^n_A$, via $f$, and $\lambda_{ij} \in k$, then $\sum_{i,j = 1}^n \lambda_{ij} v_{ij} \in I_f$.
\end{proposition}
\begin{proof}
Let $k^{(L)}$ for a set $I$ be a vector space over $k$. We define the $A-$module structure on $(k^{(L)})^n$ using $f: A \to M_n(k)$. Actually, as an $A-$module, $(k^{(L)})^n \cong (k^n_A)^{(L)}$ with $k^n_A \cong M$. Since $M$ is a finitely presented $A-$module, we have that:
$$\mathrm{Hom}_{A}(\bigoplus_{\ell \in L} M_\ell, \bigoplus_{\ell \in L} M_\ell) \cong \prod_{\ell_1 \in L} \bigoplus_{\ell_2 \in L} \mathrm{Hom}_{A}(M_{\ell_1}, M_{\ell_2})$$
Suppose we have an $A-$endomorphism of $(k^{(L)})^n$ given by $(M_{ij})_{i,j = 1}^n$ for $M_{ij}: k^{(L)} \to k^{(L)}$. Every $M_{ij} = ((m_{ij}^{\ell_1 \ell_2})_{\ell_2 \in L})_{\ell_1 \in L}$ such that for any $\ell_1 \in L$ only finitely many $m_{ij}^{\ell_1 \ell_2}$ are non-zero and for every $\ell_1, \ell_2 \in L$, $(m_{ij}^{\ell_1 \ell_2})_{i,j = 1}^n$ form an endomorphism of $M$ as an $A-$module.
If $\sum_{i,j = 1}^n \lambda_{ij} m_{ij} = 0$ holds for all $(m_{ij})_{i,j = 1}^n$ $A-$endomorphisms of $M$, clearly, $\sum_{i,j = 1}^n \lambda_{ij} M_{ij} = 0$ holds as well.
The claim follows by applying this observation to $k\langle (v_{ij})_{i,j = 1}^n \rangle/I_f$ as an $A-$module and its $A-$endomorphism $(\mu(v_{ij}))_{i,j = 1}^n$ , where $(\mu(v_{ij}))_{i,j = 1}^n$ is the multiplication by $v_{ij}$ from the right.
\end{proof}
\begin{theorem}\label{TEpirmorphismInvariant}
Let us have a finite acyclic quiver $Q'$ and a subquiver $Q$ with $Q_0 = Q'_0$ and $Q_1 \cup \{e'\} = Q'_1$. Suppose there is a representation $M' = (M_v, \varphi_{\alpha})_{v \in Q'_0, \varphi \in Q'_1}$ of $Q'$ over $k$ of dimension $n$ such that is brick and $\varphi_{e'}$ is not of full rank. Consider the representation $M = (M_v, \varphi_{\alpha})_{v \in Q_0, \varphi \in Q_1}$ of $Q$ over $k$. Denote $M'_{s(e')}$ and $M'_{t(e')}$ such that $\mathrm{Ker}\,\varphi_{e'} \oplus M'_{s(e')} = M_{s(e')}$ and $\mathrm{Im}\,\varphi_{e'} \oplus M'_{t(e')} = M_{t(e')}$, respectively. We write $\varphi_{e'}: \mathrm{Ker}\,\varphi_{e'} \oplus M'_{s(e')} \to \mathrm{Im}\,\varphi_{e'} \oplus M'_{t(e')} = M_{t(e')}$ as follows:
$$\varphi_{e'} = \left(\begin{array}{cc} 0 & \varphi \\ 0 & 0 \end{array}\right)$$
Provided that $\mathrm{Ker}\,\varphi_{e'}$ and $\mathrm{Im}\,\varphi_{e'}$ are invariant under all endomorphisms of $M$, there exists a ring epimorphism $f: kQ' \to M_n(k\langle X \rangle)$ such that on $kQ$ this coincides with $kQ \to M_n(k)$ giving the action of $kQ$ on $M$ with:
$$s(e') \mapsto \left(\begin{array}{ccc} I_{s(e')} & 0 & 0 \\ 0 & 0 & 0 \\ 0 & 0 & 0 \end{array} \right) \, \mathrm{and} \, t(e') \mapsto \left( \begin{array}{ccc} 0 & 0 & 0 \\ 0 & I_{t(e')} & 0 \\ 0 & 0 & 0 \end{array} \right)$$
where $I_{s(e')}$ is $\mathrm{dim}_k \,M_{s(e')} \times \mathrm{dim}_k \,M_{s(e')}$ identity matrix and $I_{t(e')}$ similarly, and:
$$e' \mapsto \left( \begin{array}{ccc} 0 & 0 & 0 \\ f_{e'} & 0 & 0 \\ 0 & 0 & 0 \end{array} \right).$$
The matrix $f_{e'}$ is given as follows:
\begin{enumerate}[(i)]
    \item $$f_{e'} = \left(\begin{array}{cc} 0 & \varphi \\ X_{21} & 0 \end{array}\right)$$
    \item if moreover $M'_{s(e')}$ is invariant under all endomorphisms of $M$, then:
    $$f_{e'} = \left(\begin{array}{cc} X_{11} & \varphi \\ X_{21} & 0 \end{array}\right)$$
    \item if moreover $M'_{t(e')}$ is invariant under all endomorphisms of $M$, then:
    $$f_{e'} = \left(\begin{array}{cc} 0 & \varphi \\ X_{21} & X_{22} \end{array}\right)$$
    \item if moreover $M'_{s(e')}$ and $M'_{t(e')}$ are invariant under all endomorphisms of $M$, then:
    $$f_{e'} = \left(\begin{array}{cc} X_{11} & \varphi \\ X_{21} & X_{22} \end{array}\right)$$
\end{enumerate}
Where $X_{ij}$ are matrices of indeterminates of appropriate dimensions and $X = X_{11} \cup X_{21} \cup X_{22}$.
\end{theorem}
\begin{proof}
We only prove the first assertion as the other ones follow in an analogous way.

Denote $f': kQ' \to M_n(k)$ defining the action of $kQ'$ on $M'$. By Proposition \ref{PIdealRelations} above, the relation from $f'(e')$ has, without loss of generality, this form:
$$\left(\begin{array}{cc} 0 & \varphi \\ 0 & 0 \end{array}\right) \left(\begin{array}{cc} s_{11} & s_{12} \\ s_{21} & s_{22} \end{array}\right) - \left(\begin{array}{cc} t_{11} & t_{12} \\ t_{21} & t_{22} \end{array}\right) \left(\begin{array}{cc} 0 & \varphi \\ 0 & 0 \end{array}\right) = \left(\begin{array}{cc} \varphi s_{21} & \varphi s_{22} - t_{11} \varphi \\ 0 & -t_{21} \varphi \end{array}\right)$$
Where $s_{ij}$ and $t_{ij}$ correspond to components of $k-$endomorphisms of $M_{s(e')} = \mathrm{Ker}\,\varphi_{e'} \oplus M'_{s(e')}$ and $M_{t(e')} = \mathrm{Im}\,\varphi_{e'} \oplus M'_{t(e')}$. Note that $M_v$ for all $v \in Q'_0$ and thus all maps between them are zero.

Naturally, $I_{f'} \subseteq k\langle (v_{ij})_{i,j=1}^n \rangle$ can be seen as being contained in $I_f \subseteq k\langle X \rangle \ast_k k\langle (v_{ij})_{i,j=1}^n \rangle$.
If $\mathrm{Ker}\,\varphi_{e'}$ and $\mathrm{Im}\,\varphi_{e'}$ are invariant under all endomorphisms of $M$, by Proposition \ref{PIdealRelations}, we obtain that $s_{12}, t_{12} \in I_f$ as they lie in the ideal $I_{kQ \to M_n(k)}$. SO the relation from $f(e')$ amounts to the following: 
$$\left(\begin{array}{cc} 0 & \varphi \\ X & 0 \end{array}\right) \left(\begin{array}{cc} s_{11} & 0 \\ s_{21} & s_{22} \end{array}\right) - \left(\begin{array}{cc} t_{11} & 0 \\ t_{21} & t_{22} \end{array}\right) \left(\begin{array}{cc} 0 & \varphi \\ X & 0 \end{array}\right) =$$
$$= \left(\begin{array}{cc} \varphi s_{21} & \varphi s_{22} - t_{11} \varphi \\ X_{21} s_{11} - t_{22} X_{21} & -t_{21} \varphi \end{array}\right)$$
Since $I_{f'}$ contains $v_{ii} - v_{jj}, v_{i'j'}$ for all $i,j,i' \neq j' \in \{1, \dots, n\}$ and $I_f$ contains the generators of $I_{f'}$, we have that $v_{ii} - v_{jj}, v_{i'j'} \in I_f$ for all $i,j,i' \neq j' \in \{1, \dots, n\}$. Finally, using the element $X_{21} s_{11} - t_{22} X_{21} \in I_f$, we derive that $x v_{ii} - v_{ii} x \in I_f$ for every $x \in X$ and $i = 1, \dots, n$. By Theorem \ref{TEpimorphismIdeal}, $f$ is an epimorphism.
\end{proof}

\begin{remark}
The theorem above is the aforementioned generalization of Theorem \ref{Theorem20}. Theorem \ref{Theorem20}, obviously without the reference to universal localisations, is a special case thereof for the brick $B$, whose all $k-$linear subspaces are invariant under all its endomorphisms as they are trivial, corresponding to $M$ and $\varphi_{e'}$ being zero.
\end{remark}

There is also a way how to construct ring epimorphisms from a path algebra to matrices over noncommutative polynomials by adding a vertex rather than an edge in a similar manner to Theorem \ref{Theorem20}, which can be easily proved using Theorem \ref{TEpimorphismIdeal}. Actually, this proposition can be also generalized in terms of constructing universal localisations as we show in the next subsection.
\begin{proposition}\label{PSimpleGluingEpimorphisms}
Let $Q$ be a finite acyclic quiver and $M = (M_v, \varphi_{\alpha})_{v \in Q_0, \varphi \in Q_1}$ a representation of $Q$ over $k$ of dimension $n$ that is brick and such that there exists $w \in Q_0$ with $\mathrm{dim}_k M_w = m > 1$. We define a quiver $Q$ with $Q'_0 = Q_0 \cup \{w'\}$, $Q'_1 = Q_1 \cup \{e'\}$, $s_{Q'}(e') = w', t_{Q'}(e') = w$ and $s_{Q'}(e) = s_Q (e), t_{Q'}(e) = t_Q (e)$ for all $e \in Q_1$. Denote $f: kQ \to M_n(k)$ defining the action of $kQ$ on $M$. There is a ring epimorphism $g: kQ' \to M_{n+1}(k\langle x_1, \dots, x_{m - 1}\rangle)$ defined as follows:
$$g(a) = \left(\begin{array}{cc} f(a) & 0 \\ 0 & 0 \end{array}\right) \,\mathrm{for}\, a \in kQ, \, g(w') = \left(\begin{array}{cc} 0 & 0 \\ 0 & 1 \end{array}\right),\, g(e') = \left(\begin{array}{cc} 0 & X \\ 0 & 0 \end{array}\right)$$
where $X = (1 \, x_1 \dots x_{m-1} \, 0 \dots 0)^T$.
\end{proposition}
\begin{proof}
Consider the ideal $I_g$ of $k\langle x_1, \dots, x_{m - 1}\rangle \ast_k k\langle (v_{ij})_{i,j = 1}^{n+1}\rangle$. Since $M$ is brick, we already have that $v_{ii} - v_{jj}, v_{i'j'} \in I_g$ for all $i,j,i' \neq j' \in \{1, \dots, n\}$. The relation from $g(w')$ gives us that $v_{ij} = 0$ for all $i \neq j$ and $i = n+1$ or $j = n+1$.

Therefore, the relation arising from $g(e')$ reads as follows:
$$\left(\begin{array}{cc} 0 & X \\ 0 & 0 \end{array}\right)\left(\begin{array}{cc} v_{11} \cdot I_n & 0 \\ 0 & v_{(n+1)(n+1)} \end{array}\right) - \left(\begin{array}{cc} v_{11} \cdot I_n & 0 \\ 0 & v_{(n+1)(n+1)} \end{array}\right) \left(\begin{array}{cc} 0 & X \\ 0 & 0 \end{array}\right)$$
This yields that $X v_{(n+1)(n+1)} = v_{11} \cdot I_n X$, which can be written as $$(v_{(n+1)(n+1)} \, v_{(n+1)(n+1)}x_1 \dots v_{(n+1)(n+1)}x_{m-1} \, 0 \dots 0)^T -$$ $$(v_{11} \, v_{11}x_1 \dots v_{11}x_{m-1} \, 0 \dots 0)^T \in I_g.$$
We obtain that $v_{(n+1)(n+1)} - v_{11}$, and we employ this to show that $v_{ii}$ commute with $x_j$ for all $i = 1, \dots, n+1$ and $j = 1, \dots, m-1$. The map $g$ is an epimorphism by Theorem \ref{TEpimorphismIdeal}.
\end{proof}

\subsection{Extending universal localisations}
In this subsection, we explore possible generalizations of Theorem \ref{Theorem20} by extending universal localisations from smaller to larger path algebras. For this end, we make use of the matrix reduction functor, and we show that, due to their similar universal property, matrix algebras over $k\langle x_1, \dots, x_n \rangle$ arise naturally when dealing with certain universal localisations of path algebras.

\begin{definition}[Matrix reduction functor, section 0.2 in \cite{cohn2006free}]\label{DMatrixReductionFunctor}
    There is a left adjoint to the $n \times n$ matrix functor $M_n(-): k-\mathsf{Alg} \to k-\mathsf{Alg}$; it is denoted $\sqrt[n]{-}: k-\mathsf{Alg} \to k-\mathsf{Alg}$ and referred to as the \textit{$n-$matrix reduction functor}.
\end{definition}
The adjunction gives us a functorial isomorphism $\mathrm{Hom}_{k-\mathsf{Alg}}(A,M_n(\sqrt[n]{B})) \cong \mathrm{Hom}_{k-\mathsf{Alg}}(\sqrt[n]{A},B)$ for all $A, B \in k-\mathsf{Alg}$. It follows that there is a map $\mu_{A, n}: A \to M_n(\sqrt[n]{A})$ such that for any $f: A \to M_n(B)$ there is a unique map $f': \sqrt[n]{A} \to B$ with $f = M_n(f')\mu_{A, n}$. This situation can be summarized by the following commutative diagram:
$$\begin{tikzcd}
\mathrm{Hom}_{k-\mathsf{Alg}}(A,M_n(\sqrt[n]{A})) \arrow[']{d}{\mathrm{Hom}_{k-\mathsf{Alg}}(A,M_n(f'))} \arrow{r}{\cong} & \mathrm{Hom}_{k-\mathsf{Alg}}(\sqrt[n]{A},\sqrt[n]{A}) \arrow{d}{\mathrm{Hom}_{k-\mathsf{Alg}}(\sqrt[n]{A},f')}\\
\mathrm{Hom}_{k-\mathsf{Alg}}(A,M_n(B)) \arrow{r}{\cong} & \mathrm{Hom}_{k-\mathsf{Alg}}(\sqrt[n]{A},B)
\end{tikzcd}$$
Actually, it is possible to describe $\sqrt[n]{A}$ and $\mu_{A,n}$ rather simply. The ring $M_n(\sqrt[n]{A})$ is isomorphic to $M_n(K) \ast_k A$; $\sqrt[n]{A}$ is isomorphic to the centralizer of $e_{ij}$ in $M_n(K) \ast_k A$, which comprises of elements of form $\sum_{v} e_{iv} a e_{vj}$ for $a \in A$ (compare Theorem 0.2.3 in \cite{cohn2006free}); $\mu_{A,n}(a) = \sum_{ij} e_{ij} \cdot \left(\sum_{v} e_{iv} a e_{vj}\right)$ for all $a \in A$ and, given $f: A \to M_n(B)$, the $f': \sqrt[n]{A} \to B$ such that $f = M_n(f) = $ is defined simply as $f\left(\sum_{v} e_{iv} a e_{vj}\right) = (f(a))_{ij}$.
\begin{lemma}\label{LLocalisationAsFactor}
Suppose that $\ell_\Sigma: A \to A_\Sigma$ is the universal localisation in a set of maps between finitely generated projectives over $A$. Let $f: A \to B$ be a map such that there is a map $g: B \to A_\Sigma$ such that $\ell_\Sigma = gf$, then there exists a surjective map $h: B_{\Sigma \otimes_A B} \to A_\Sigma$ such that $g = h \ell_{\Sigma \otimes_A B}$, where $\Sigma \otimes_A B = \{\sigma \otimes_A B \,|\, \sigma \in \Sigma\}$ and $\ell_{\Sigma \otimes_A B}: B \to B_{\Sigma \otimes_A B}$ is the universal localisation of $B$ in $\Sigma \otimes_A B$.
\end{lemma}
\begin{proof}
Due to the universal property of the universal localisation $\ell_\Sigma$, we have a unique map $i: A_\Sigma \to B_{\Sigma \otimes_A B}$ with the property that $\ell_{\Sigma \otimes_A B} f = i \ell_\Sigma$ as every $\sigma \otimes_A B_{\Sigma \otimes_A B}$, $\sigma \in \Sigma$, is invertible in $B_{\Sigma \otimes_A B}$. Similarly, there is a exists a unique map $j: B_{\Sigma \otimes_A B} \to A_\Sigma$ extending the map $g$ since every $(\sigma \otimes_A B) \otimes A_\Sigma = \sigma \otimes_A A_\Sigma$, $\sigma \in \Sigma$, is invertible. We can calculate that $j i \ell_\Sigma = j \ell_{\Sigma \otimes_A B} f = gf = \ell_\Sigma$. By universal property of $\ell_\Sigma$, only the identity map on $A_\Sigma$ satisfies that $\ell_\Sigma = m \ell_\Sigma$ for $m: A_\Sigma \to A_\Sigma$; therefore, $ji = \mathsf{id}_{A_\Sigma}$. This implies that $j$ is surjective.
\end{proof}

\begin{corollary}
Suppose that $\ell_\Sigma: A \to A_\Sigma$ is a universal localisation in $\Sigma$ such that $A_\Sigma \cong M_n(B)$ for some $B$ and $n$, then $A_\Sigma$ is isomorphic to a factor of $M_n(\sqrt[n]{A})_{\Sigma'}$ where $\Sigma' = \{\sigma \otimes_A M_n(\sqrt[n]{A}) \, | \, \sigma \in \Sigma\}$.
\end{corollary}

\begin{proposition}\label{PMatrixReductionEpimorphism}
Let $f: A \to M_n(B)$ be map of $k-$algebras, then the following are equivalent:
\begin{enumerate}[(i)]
    \item The map $f$ is an epimorphism.
    \item The dominion of $f(A)$ in $M_n(B)$ contains all $e_{ij}$ for $1 \leq i, j \leq n$ and the corresponding map from $\sqrt[n]{A}$ to $B$ is an epimorphism.
    \item The map $\sqrt[n]{f}: \sqrt[n]{A} \to \sqrt[n]{M_n(B)}$ is an epimorphism.
\end{enumerate}
\end{proposition}
\begin{proof}
$(i) \Leftrightarrow (ii):$ Provided that $f: A \to M_n(B)$ is an epimorphism, then the dominion of $f(A)$ in $M_n(B)$ contains all of $M_n(B)$. The fact that $f: A \to M_n(B)$ is an epimorphism implies that $\mathrm{Hom}_{k-\mathsf{Alg}}(f,M_n(C))$ is injective for any $C$. However, the adjunction implies that so is $\mathrm{Hom}_{k-\mathsf{Alg}}(f',C)$ where $f'$ corresponds to $f$ under the natural isomorphism $\mathrm{Hom}_{k-\mathsf{Alg}}(A,M_n(B)) \cong \mathrm{Hom}_{k-\mathsf{Alg}}(\sqrt[n]{A},B))$. Therefore, $f': \sqrt[n]{A} \to B$ is an epimorphism.\\\\
Suppose that the dominion of $f(A)$ in $M_n(B)$ contains all $e_{ij}$ for $1 \leq i, j \leq n$ and the corresponding map $f': \sqrt[n]{A} \to B$ is an epimorphism. Let us have $C$ and two maps $g_1, g_2: M_n(B) \to C$ such that $g_1 f = g_2 f$. We observe that images of $e_ij$ under $g_1$ and $g_2$ are equal as they lie in the dominion of $f(A)$ in $M_n(B)$ and they induce a matrix ring structure on $C$ represented by an isomorphism $h: C \to M_n(C')$ with $C'$ being the centralizer of all $g_1(e_{ij}) = g_2(e_{ij})$. Moreover, $h g_1$ and $h g_2$ restrict to maps $g'_1, g'_2$ between $B$ and $C'$ because $B$ as centralizer of $e_{ij}$ maps to the centralizer of $g_1(e_{ij})$ or $g_2(e_{ij})$ under $g_1$ or $g_2$, respectively. In other words, $M_n(g'_1) = h g_1$ and similarly for $g_2$. This allows us to employ the adjunction: $M_n(g'_1) f = M_n(g'_2) f$ means that $g'_1 f' = g'_2 f'$, but, as $f'$ is an epimorphism, $g'_1$ need to equal $g'_2$. It follows easily that $g_1 = g_2$.\\\\
$(i) \Leftrightarrow (iii):$ Let us have a homomorphism $g: M_n(B) \to C'$. There is a $C$ a $k-$algebra such that $C' \cong M_n(C)$. Hence, $f: A \to M_n(B)$ is an epimorphism is equivalent to: $$\mathrm{Hom}_{k-\mathsf{Alg}}(f,M_n(C)): \mathrm{Hom}_{k-\mathsf{Alg}}(A, M_n(C)) \to \mathrm{Hom}_{k-\mathsf{Alg}}(M_n(B),M_n(C))$$ being surjective for every $k-$algebra $C$, which, by the adjuction, is equivalent to $\mathrm{Hom}_{k-\mathsf{Alg}}(\sqrt[n]{f},C): \mathrm{Hom}_{k-\mathsf{Alg}}(\sqrt[n]{A}, C) \to \mathrm{Hom}_{k-\mathsf{Alg}}(\sqrt[n]{M_n(B)},C)$ being surjective for every $k-$algebra $C$. The last statement is equivalent to $\sqrt[n]{f}: \sqrt[n]{A} \to \sqrt[n]{M_n(B)}$ being an epimorphism.
\end{proof}

\begin{theorem}\label{TExtendingLocalisation}
Let $i: A \to A'$ be a homomorphism, and let $\ell_\Sigma: A \to A_\Sigma$ is a universal localisation such that $A_\Sigma \cong M_n(B)$. We claim that $A'_{\Sigma'}$ with $\Sigma' = \{\sigma \otimes_A A' \,|\, \sigma \in \Sigma\}$ is a factor of the universal localisation of $M_n(\sqrt[n]{A'})$ in induced maps arising only from relations in representation of $A_\Sigma$ as a factor of the universal localisation $M_n(\sqrt[n]{A})$ in in induced maps.
\end{theorem}
\begin{proof}
At first, we denote $I$ the ideal of $M_n(\sqrt[n]{A})$ such that $M_n(\sqrt[n]{A})/I \cong A_\Sigma$. We claim that there exists the following commutative diagram:
$$\begin{tikzcd}
A \arrow[']{d}{\mu_{A, n}} \arrow{r}{f} & A' \arrow{d}{\mu_{A',n}}\\
M_n(\sqrt[n]{A}) \arrow[']{d}{\ell_{\Sigma_n}} \arrow{r}{i_1 = M_n(\sqrt[n]{f})} & M_n(\sqrt[n]{A'}) \arrow{d}{\ell_{\Sigma'_n}}\\
M_n(\sqrt[n]{A})_{\Sigma_n} \arrow{d}[']{\pi_i} \arrow{r}{i_2 = (i_1)_{\Sigma_n}} & M_n(\sqrt[n]{A'})_{\Sigma'_n} \arrow{d}{\pi_{I'}}\\
M_n(\sqrt[n]{A})_{\Sigma_n}/I \arrow{r}{i_3} & M_n(\sqrt[n]{A'})_{\Sigma'_n}/I'
\end{tikzcd}$$
Where $\Sigma_n = \{\sigma \otimes_A M_n(\sqrt[n]{A}) \,|\, \sigma \in \Sigma\}$, $\Sigma'_n = \{\sigma \otimes_A M_n(\sqrt[n]{A'}) \,|\, \sigma \in \Sigma\}$ and $I'$ is the ideal generated by images of elements of $I$ under $i_2$. Each horizontal map below $i$ is obtained using the universal property of the vertical map on the left applied to the composition of the horizontal map above and the vertical map on the right. Moreover, all horizontal maps between the matrix algebras respect the matrix structure. In other words, they are $M_n$ of some map between the centralizers of $e_{ij}$.

We will prove two statements:
\begin{enumerate}[(i)]
    \item The map $\pi_{I'} \circ \ell_{\Sigma'_n} \circ \mu_{A', n}: A' \to M_n(\sqrt[n]{A'})_{\Sigma'_n}/I'$ from the commutative diagram map is an epimorphism.
    \item The universal localisation $\ell_{\Sigma'}: A' \to A'_{\Sigma'}$ factors through said map $A' \to M_n(\sqrt[n]{A'})_{\Sigma'_n}/I'$.
\end{enumerate}
Together, these two statements imply that said map $A' \to M_n(\sqrt[n]{A'})_{\Sigma'_n}/I'$ is the universal localisation of $A'$ in $\Sigma'$.

The fact that the map $A \to M_n(\sqrt[n]{A})_{\Sigma}/I$ from the commutative diagram map is an epimorphism yields that $e_{ij}$ for $1 \leq i, j \leq n$ lie in the dominion of $A$ in $M_n(\sqrt[n]{A})_{\Sigma}/I$. Combining it with the fact that homomorphic image of a dominion of a subring lies in the dominion of the homomorphic image of the subring gives us that $e_{ij}$ for $1 \leq i, j \leq n$ lie in the dominion of $i_3 \circ \pi_i \circ \ell_{\Sigma_n} \circ \mu_{A, n}$ in $M_n(\sqrt[n]{A'})_{\Sigma'_n}/I'$, specifically also in the dominion of image of $A'$. We recall that taking universal localisation or a factor by an ideal effectively commutes with forming a matrix algebra; this gives that the said map $A' \to M_n(\sqrt[n]{A'})_{\Sigma'_n}/I'$ corresponds, by abuse of notation, to a map $\sqrt[n]{A'} \to \sqrt[n]{A'}_{\Sigma'_n}/I'$, which is an epimorphism. We have now established that the said map $A' \to M_n(\sqrt[n]{A'})_{\Sigma'_n}/I'$ is an epimorphism.

Finally, we ought to show that the universal localisation $\ell_{\Sigma'}: A' \to A'_{\Sigma'}$ factors through said map $A' \to M_n(\sqrt[n]{A'})_{\Sigma'_n}/I'$. Consider the following commutative diagram:
$$\begin{tikzcd}
A \arrow{r}{} \arrow{d}{} & M_n(\sqrt[n]{A})_{\Sigma_n} \arrow{r}{} \arrow{d}{} & M_n(\sqrt[n]{A})_{\Sigma_n}/I \arrow[dotted]{d}{}\\
A' \arrow{r}{} & M_n(\sqrt[n]{A'})_{\Sigma'_n} \arrow{r}{} & A'_{\Sigma'_n}
\end{tikzcd}$$
The dotted map results from using the universal property of $\ell_\Sigma$ to the map $A \to M_n(\sqrt[n]{A})_{\Sigma_n} \to M_n(\sqrt[n]{A'})_{\Sigma'_n} \to A'_{\Sigma'}$, which is a $\Sigma-$inverting map. Commutativity of the square on the right is due to the universal property of $A \to M_n(\sqrt[n]{A})$ and to $M_n(\sqrt[n]{A}) \to M_n(\sqrt[n]{A})_{\Sigma_n}$ being an epimorphism; since $A'_{\Sigma'}$ can be given structure of $n \times n$ matrices over some algebra, the image of $A$ fully determines the image of $M_n(\sqrt[n]{A})$, and also $M_n(\sqrt[n]{A})_{\Sigma_n}$, indeed.
From the commutative diagram, we read that images of elements of $I$ in $M_n(\sqrt[n]{A'})_{\Sigma'_n}$ need to map to zero, so the morphism $M_n(\sqrt[n]{A'})_{\Sigma'_n} \to A'_{\Sigma'}$ factors through $M_n(\sqrt[n]{A'})_{\Sigma'_n}/I'$ as $I'$ is generated by such elements.
\end{proof}
Let us have a path algebra $kQ$ with the vertices of $Q_0$ labelled $1,\dots,n$ and a dimension vector $\alpha \in \mathbb{N}_0^{Q_0}$. We will construct a map $q_{Q,\alpha}$ from $kQ$ to $M_\alpha(k\langle X_{Q,\alpha} \rangle)$ which is the algebra of $\sum_{i=1}^n \alpha_i \times \sum_{i=1}^n \alpha_i$ matrices over with:
$$X_{Q,\alpha} = \bigcup_{e \in Q_1} \{x^{(e)}_{ij} \,|\, 1 \leq i \leq \alpha_{s(e)}, 1 \leq j \leq \alpha_{s(e)}\}$$
In order to simplify working with $\sum_{i=1}^n \alpha_i \times \sum_{i=1}^n \alpha_i$ matrices over $k\langle X_{Q,\alpha} \rangle$, we think of elements of $M_\alpha(k\langle X_{Q,\alpha} \rangle)$ as of $n \times n$ block matrices such that the block at position $ij$ for $1 \leq i, j \leq n$ is an $\alpha_i \times \alpha_j$ matrix.

For each $m \in Q_0$, we set $q_{Q,\alpha}(e_m)_{ij} = I_{\alpha_m}$ for $i = j = m$ and $q_{Q,\alpha}(e_m)_{ij} = 0$, otherwise, when viewed as an $n \times n$ block matrix. For each $e \in Q_1$, we set $q_{Q,\alpha}(e)_{ij} = X^{(e)}_{Q,\alpha}$ for $i = s(e), j = s(e)$ and $q_{Q,\alpha}(e)_{ij} = 0$, otherwise; $X^{(e)}_{Q,\alpha}$ is a $\alpha_{t(e)} \times \alpha_{s(e)}$ matrix with $x^{(e)}_{ij}$ at position $ij$.

Thus, by universal property of the path algebra $kQ$, this assignment can be extended to a map from $q_{Q,\alpha}: kQ \to M_\alpha(k\langle X_{Q,\alpha} \rangle)$. This map factors through $\mu_{A,n}$ via $q'(Q,\alpha)$, and $q'_{Q, \alpha}$ is surjective as the entries of matrices of image of $q_{Q,\alpha}$ generate $k\langle X_{Q,\alpha} \rangle)$ as a $k-$algebra by discussion below Definition \ref{DMatrixReductionFunctor}.

\begin{proposition}\label{PCanonicalForm}
Suppose $f: kQ \to R'$ is a map and $S$ is a module over $R$ such that $e_i (kQ) \otimes_{kQ} R' \cong S^{\alpha_i}$ for all $1 \leq i \leq n$, then there exists $R$ such that $R' \cong M_n(R)$. The map $f$ viewed now as going from $kQ$ to $M_n(R')$ uniquely factors via $M_n(f')$ through $q_{Q, \alpha}$.
\end{proposition}
\begin{proof}
Since $kQ \cong e_1 (kQ) \oplus \dots \oplus e_n (kQ)$ as a right module over itself and $kQ \otimes_{kQ} R' \cong R'$, it follows that $R' \cong S^{\alpha_1} \oplus \dots \oplus S^{\alpha_n}$. Hence, there is an isomorphism $R' \to M_\alpha(\mathrm{End}_{R'}(S))$. Set $R = \mathrm{End}_{R'}(S)$, and let us work with $f$ as with a map from $kQ \to M_n(R)$.
Without loss of generality, we can assume that $f(e_m)_{ij} = I_{\alpha_m}$ for $i = j = m$ and $f(e_m)_{ij} = 0$, otherwise. For an arrow $e \in Q_1$ from $s(e)$ to $t(e)$, this means that $f(e)_{ij} \neq 0$ if and only if $i = t(e)$ and $j = s(e)$ and $f(e)_{(t(e))(s(e))}$ is an $\alpha_{t(e)} \times \alpha_{s(e)}$ matrix. We define $f' k\langle X_{Q, \alpha} \rangle \to R$ such that $x_{ij}^{(e)} \mapsto (f(e)_{(t(e))(s(e))})_{ij}$; we obtain $f = M_n(f') \circ q_{Q, \alpha}$. Uniqueness thereof follows from the same universal property of $k\langle X_{Q, \alpha} \rangle$.
\end{proof}

Using the theorem above, we can formulate another generalization of Theorem \ref{Theorem20}, which in extending universal localisation from a smaller path algebra to a larger one.

\begin{proposition}\label{PExtendingExceptional}
Let $Q$ be a finite acyclic quiver and $M \in \mathsf{mod}-kQ$ be an exceptional module of dimension vector $\alpha$, then the map $f: kQ \to M_\alpha(k)$ is a universal localisation in $\Sigma$ set of finitely generated maps between projectives. Let $Q'$ be a finite acyclic quiver such that $Q_0 = Q'_0$ and arrows in $Q$ arise from paths in $Q'$; $kQ$ is a naturally subalgebra of $kQ'$. Then, the universal localisation of $kQ'$ in the set maps induced from $\Sigma$ has the following form:
$$kQ'_{\Sigma'} \cong M_\alpha (k\langle X_{Q', \alpha} \rangle)/I$$
where $I$ is generated by $q_{Q', \alpha}(a) - \iota f(a)$ for all $a \in kQ$ viewed as a subalgebra of $kQ'$ with $\iota$ being $M_\alpha$ of the map $k \to k\langle X_{Q', \alpha} \rangle$.
\end{proposition}
\begin{proof}
This is an easy consequence of Lemma \ref{Lemma15}, Theorem \ref{TExtendingLocalisation}, Proposition \ref{PCanonicalForm}, and the fact that $q_{Q',\alpha}$ factors through $\mu_{kQ', \alpha}$ via a surjective map.
\end{proof}

As mentioned above, Proposition \ref{PSimpleGluingEpimorphisms} can be generalized in the context of universal localisations; this idea is formalized in the following proposition.

\begin{proposition}
Suppose that we have finite acyclic quivers $Q_1$ and $Q_2$ and universal localisations $f_{\Sigma_1}: kQ_1 \to M_{\alpha_1}(k\langle x_1, \dots, x_{n_1} \rangle)$ and $f_{\Sigma_2}: kQ_2 \to M_{\alpha_2}(k\langle x_1, \dots, x_{n_2} \rangle)$ for $\Sigma_1$ and $\Sigma_2$ sets of maps between finitely generated projectives over $kQ_1$ and $kQ_2$, respectively.

Let $Q$ be a quiver such that $Q_0 = (Q_1)_0 \cup (Q_2)_0$, $Q_1$ such that for all $e \in Q'_1$ either is $s(e) \in (Q_1)_0, t(e) \in (Q_2)_0$ or $s(e) \in (Q_2)_0, t(e) \in (Q_1)_0$. Denote $\Sigma$ the set of maps between finitely generated projectives that correspond to maps in $\Sigma_1$ and $\Sigma_2$.

Then $kQ_\Sigma$ is Morita equivalent to representations of a quiver $Q'$ with two vertices $v_1$ and $v_1$ with $n_1$ loops for $v_1$, $n_2$ loops for $v_2$, $\alpha_1(s(e))\cdot \alpha_2(t(e))$ arrows from $v_1$ to $v_2$ for each $e \in Q_1$ with $s(e) \in (Q_1)_0$ and $t(e) \in (Q_2)_0$, and $\alpha_2(s(e))\cdot \alpha_1(t(e))$ arrows from $v_2$ to $v_1$ for each $e \in Q_1$ with $s(e) \in (Q_2)_0$ and $t(e) \in (Q_1)_0$.
\end{proposition}
\begin{proof}
Denote $I_1$ and $I_2$ the ideal generated by $v \in (Q_2)_0$ and $v \in (Q_1)_0$, respectively. We have that $kQ/I_j \cong kQ_j$ and that:
$$kQ_\Sigma/\langle I_j \rangle_{kQ_\Sigma} \cong (kQ_j)_{\Sigma_j} \cong M_{\alpha_j}(k\langle x_1, \dots, x_{n_j} \rangle)$$
for $j = 1, 2$.
Given a representation $M = (M_v, \varphi_e)_{v \in Q_0, e \in Q_1}$ of $Q$ over $k$ that is a restriction of module over $kQ_\Sigma$, we find that $(M_v, \varphi_e)_{v \in (Q_2)_0, e \in (Q_2)_1}$ is a restriction of a module over $(kQ_1)_{\Sigma_1}$ and, therefore, corresponds to a representation $(M_1, \lambda^M_1, \dots, \lambda^M_{n_1})$ of $k\langle x_1, \dots, x_{n_1}\rangle$. Analogously, $(M_v, \varphi_e)_{v \in (Q_2)_0, e \in (Q_2)_1}$ corresponds to  $(M_2, \mu^M_1, \dots, \mu^M_{n_2})$. Also, $M_v = (M_1)^{\alpha_1(v)}$ and $M_v = (M_2)^{\alpha_2(v)}$ for $v \in (Q_1)_0$ and $v \in (Q_2)_0$, respectively.

For $e \in Q'_1$ with source in $(Q_1)_0$ and target in $(Q_2)_0$, we have that $\varphi_e: M_1^{\alpha_1(s(e))} \to M_2^{\alpha_2(t(e))}$, which means it is determined by a $\alpha_2(t(e)) \times \alpha_1(s(e))$ matrix of maps from $M_1 \to M_2$, similarly for $e' \in Q'_1$ with source in $(Q_2)_0$. Therefore, $M$ corresponds to a representation $F(M)$ of $Q'$ given by $F(M)_{v_1} = M_1$, $F(M)_{v_2} = M_2$, the loops are given by $\lambda_1, \dots, \lambda_{n_1}$ and $\mu_1, \dots, \mu_{n_1}$ and other arrows are given by maps from $M_1$ to $M_2$ or vice versa that form components of respective $e \in Q'_1$ as discussed above.

By considering the components pertaining to $Q_1$ and $Q_2$, we obtain that any possible homomorphism $f$ between $M$ and $N$ has the following form $f = ((f_1 I_{\alpha_1(v)})_{v \in (Q_1)_0},$ $(f_2 I_{\alpha_2(v)})_{v \in (Q_2)_0})$ such that $f_1$ commutes with all $\lambda$ and $f_2$ commutes with all $\mu$. We also have that $f_2 I_{\alpha_2(t(e))} \cdot \varphi^M_e = \varphi^N_e \cdot f_1 I_{\alpha_1(s(e))}$ for every $e' \in Q'_1$ with source in $(Q_1)_0$ and target in $(Q_2)_0$. Similarly, it holds for $e' \in Q'_1$ from a vertex in $(Q_2)_0$ and to a vertex in $(Q_1)_0$. This gives us that $f$ commutes with $M_1$ to $M_2$ or vice versa representing arrows in $Q'$. We set $F(f)_{v_1} = f_1$ and $F(f)_{v_2} = f_2$.

We have constructed a functor $F$ from $\mathsf{Mod}-kQ_\Sigma$ to $\mathsf{Mod}-kQ'$. It is easy to show that the functor $F$ is well-defined, fully faithful and essentially surjective.
\end{proof}

\section*{Acknowledgments}
This text was prepared as in the course of the author's work on a doctoral thesis at Faculty of Mathematics and Physics, Charles University advised by Jan \v{S}\v{t}ov\'{i}\v{c}ek, to whom the author wishes to thank for his help and support.

The author was partially supported from grant \textit{GA\v{C}R 17-23112S} of the Czech Science Foundation.

\renewcommand{\bibname}{References}
\bibliographystyle{alpha}
\bibliography{main}

~\\[0.05cm]
\noindent
\textsc{Charles University, Faculty of Mathematics and Physics, Department of Algebra, Sokolovska 83, 186 75 Praha, Czech Republic}\\\\
\textit{Email address:} \texttt{jakub.kopriva@outlook.com}

\end{document}